\documentclass[11pt]{amsart}

\usepackage{epsfig}

\usepackage{graphicx}
\usepackage{amscd}
\usepackage{amsmath}
\usepackage{amsfonts}
\usepackage{psfrag}
\usepackage{amssymb}
\usepackage{verbatim}
\usepackage{comment}

\newtheorem{theorem}{Theorem}[section]
\theoremstyle{plain}

\newtheorem{lemma}[theorem]{Lemma}

\newtheorem{prop}[theorem]{Proposition}

\numberwithin{equation}{section}

\def\Xxi{{\mathfrak X}_\zeta}
\def\Xx{{\mathfrak X}}

\def\One{{1\!\!1}}

\def\Lip{{\rm Lip}}

\def\Hk{{\mathcal H}}

\def\wtil{\widetilde}
\def\half{\frac{1}{2}}

\newcommand{\gam}{\gamma}
\newcommand{\om}{\omega}

\def\Om{\Omega}

\newcommand{\Ups}{\Upsilon}
\newcommand{\Gam}{\Gamma}
\newcommand{\sig}{\sigma}

\newcommand{\R}{{\mathbb R}}

\newcommand{\Z}{{\mathbb Z}}
\newcommand{\C}{{\mathbb C}}

\def\N{{\mathbb N}}

\newcommand{\Nat}{{\mathbb N}}

\def\A{{\mathcal A}}

\def\Sf{{\sf S}}

\def\be{\begin{equation}}
\def\ee{\end{equation}}
\newcommand{\Ek}{{\mathcal E}}
\newcommand{\Fk}{{\mathcal F}}

\newcommand{\eps}{{\varepsilon}}

\def\ov{\overline}

\newcommand{\const}{{\rm const}}

\def\card{{\rm card}}

\def\ve1{\vec{1}}

\def\Ak{{\mathcal A}}

\def\Bu{B}

\begin{document}

\title[On ergodic averages for parabolic product flows]{On  ergodic averages for parabolic product flows}

\author{Alexander I. Bufetov}
\address{Alexander I. Bufetov\\ 
Aix-Marseille Universit{\'e}, CNRS, Centrale Marseille, I2M, UMR 7373}
\address{39 rue F. Joliot Curie Marseille France }
\address{Steklov  Mathematical Institute of RAS, Moscow}
\address{Institute for Information Transmission Problems, Moscow}
\address{National Research University Higher School of Economics, Moscow}
\address{The Chebyshev Laboratory, Saint-Petersburg State University, Saint-Petersburg, Russia}

\email{alexander.bufetov@univ-amu.fr, bufetov@mi.ras.ru}
\author{Boris Solomyak }
\address{Boris Solomyak\\ Department of Mathematics,
Bar-Ilan University, Ramat-Gan, Israel}
\email{bsolom3@gmail.com}

\begin{abstract}  We consider a direct product of a suspension flow over a substitution dynamical system and an arbitrary ergodic flow and give quantitative  estimates for the speed of convergence for ergodic integrals of such systems.
Our argument relies on new uniform estimates of the spectral measure for suspension flows over substitution dynamical systems. The paper answers a question by Jon Chaika.
\end{abstract}



\maketitle

\thispagestyle{empty}

\section{Introduction}

Parabolic dynamical systems are characterized by a ``slow'' chaotic behavior: whereas for hyperbolic systems nearby trajectories diverge exponentially, for parabolic ones they diverge polynomially in time. Classical examples include the horocycle flows and translation flows on flat surfaces of higher genus. Substitution dynamical systems and suspension flows over them also fall into this category. Due to their simple-to-describe combinatorial framework and many connections, e.g.\ with number theory and automata theory, they have  provided a ``testing ground'' for new methods. Their spectral theory has been actively studied, but many natural questions remain open.

We refer the reader to \cite{Queff,Siegel} for a detailed background, but recall the basic definitions briefly.
Let $\Ak=\{1,\ldots,m\}$ be a finite alphabet; we denote by $\Ak^+$ the set of finite (non-empty) words in $\Ak$.  A {\em substitution} is a map $\zeta:\, \Ak \to \Ak^+$, which is extended to an action on $\Ak^+$ and $\Ak^\N$ by concatenation. (Using a different language, this is a morphism of a free semigroup with $\Ak$ being a set of free generators.)
The {\em substitution space}, denoted $X_\zeta$, is a subset of $\Ak^\Z$ consisting of all two-sided infinite sequences $x$ with the property that for every $n\in N$, the word, or block, $x[-n,n]$ occurs as a subword in $\zeta^k(a)$ for some $k\in \N$ and $a\in \Ak$. It is clearly closed (in the discrete  product topology) and shift-invariant; thus we obtain a topological {\em substitution dynamical system} $(X_\zeta, T_\zeta)$, where $T_\zeta$ denotes the left shift restricted to $X_\zeta$. The {\em substitution matrix} is defined by 
$$
\Sf_\zeta (i,j) = \mbox{number of symbols}\ i\ \mbox{in the word}\ \zeta(j).
$$
This is a non-negative integer $m\times m$  matrix, which provides the  {\em abelianization} of the free semigroup morphism $\zeta$. Assume that $\zeta$ is {\em primitive}, that is,
some power of $\Sf_\zeta$ has only positive emtries. In this case the $\Z$-action  $(X_\zeta,T_\zeta)$  is minimal and uniquely ergodic, with a unique invariant Borel probability measure $\mu$. We also assume that $\zeta$ is {\em aperiodic}, i.e.\ the system has no periodic points, excluding the trivial case of $X_\zeta$ finite. Denote by $\theta_j$, $j\ge 1$, the eigenvalues of $\Sf_\zeta$ ordered by magnitude:
$$
\theta_1 > |\theta_2| \ge \ldots
$$
The famous ``Pisot substitution conjecture'' asserts that if $|\theta_2|<1$, then the measure-preserving system has pure discrete spectrum. (This condition is equivalent to $\theta_1$ being a Pisot number and the characteristic polynomial of $\Sf_\zeta$ being irreducible.)
This is known only in the two-symbol case \cite{BD,HS}, although there has been a lot of progress recently, see \cite{ABBLS}. In any case, such substitution systems have a large discrete component: they have a factor which is an irrational translation on an $(m-1)$-dimensional torus. 

Along with the substitution $\Z$-action, it is natural to study suspension flows over them. We only consider piecewise-constant roof functions. More precisely, for  a strictly positive vector $\vec{s} = (s_1,\ldots,s_m)$ we consider the suspension flow over $T_\zeta$, with the piecewise-constant roof function, equal to $s_j$ on the cylinder set $[j]$. The resulting space will be denoted by $\Xxi^{\vec{s}}$, the unique invariant measure for our suspension flow by ${\widetilde \mu}$  and the flow
by $(\Xxi^{\vec{s}}, {\widetilde \mu}, h_t)$. 
We have, by definition,
$$
\Xxi^{\vec{s}} = \bigcup_{a\in \A}  \Xx_a,\ \ \ \mbox{where}\ \ \Xx_a= \{(x,t):\ x\in X_\zeta,\ x_0=a,\ 0 \le t \le s_a\}. 
$$
and this union is disjoint in measure. We call the system $(\Xxi^{\vec{s}}, {\widetilde \mu}, h_t)$ a {\em substitution $\R$-action}.
This flow can also be viewed as the translation action on a tiling space, with interval prototiles of length $s_j$. A special case of interest is when 
$\vec{s}$ is the Perron-Frobenius eigenvector for the transpose substitution matrix $\Sf^t_\zeta$; this corresponds to the self-similar tiling on the line, see \cite{SolTil}.


Sometimes results on spectral properties become simpler when we pass from the $\Z$-action to the $\R$-action. In particular, the condition ``$\theta_1$ is not Pisot'' is equivalent to the substitution $\R$-action being weakly mixing, in the self-similar case \cite{SolTil} and in the ``generic case'' (for Lebesgue-a.e.\ $\vec{s}$) \cite{CSa}, whereas the situation for substitution $\Z$-actions is much more complicated \cite{solomyak,FMN}. In this paper we continue the analysis of the generically weak-mixing case, assuming $|\theta_2|>1$, which was started in \cite{BuSo1}. Note that the ``borderline'' case $|\theta_2|=1$ is more subtle \cite{BBH}. We should also note that when the characteristic polynomial of $\Sf_\zeta$ is reducible, e.g.\ when $\theta_1$ is an integer, the type of spectrum is determined not just by the matrix, but also by the order of the letters in the words $\zeta(j)$, see e.g.\ \cite{Queff}.


As is well-known, weak-mixing of a system is equivalent to the  ergodicity of  the  product flow $h_t \times H_t$, where $H_t$ is an arbitrary measure-preserving ergodic flow defined  on a standard probability space $(Y, \nu)$.
We thus have 
\be \label{converge}
\lim_{R\to \infty} \frac{1}{R} \int_0^R \langle (f\otimes g)\circ(h_t\times H_t), f\otimes g\rangle\,dt=0 \  \mbox{for all}\ f\in L^2(\Xxi^{\vec{s}}, {\widetilde \mu}), g \in L^2_0(Y,\nu).
\ee
Our aim in this paper is to  give power estimates for the speed of convergence in (\ref{converge}).

On $\Xxi^{\vec{s}}$ we consider Lipschitz ``cylindrical functions," namely, functions of the form
$$
f(x,t) = \psi_{x_0},\ \ 0 \le t\le  s_{x_0},
$$
where $\psi_j \in \Lip[0,s_j]$, $j\le m$. Denote
$$
\|f\|_L := \max_{a\in \Ak} \|\psi_a\|_L,
$$
where $\|\cdot\|$ is the Lipschitz norm.

Let $f\in L^2(\Xxi^{\vec{s}},\wtil{\mu})$. By the Spectral Theorem for measure-preserving flows, there is a finite positive Borel measure $\sig_f$ on $\R$ such that
$$
\widehat{\sig}_f(-t) = \int_{-\infty}^\infty e^{2 \pi i\om t}\,d\sig_f(\om) = \langle f\circ h_t, f\rangle\ \ \ \mbox{for}\ t\in \R,
$$
where $\langle \cdot,\cdot \rangle$ denotes the inner product in $L^2$. For functions $f$ and $g$ set $(f\otimes g)(x,y):= f(x)g(y)$. 
Without loss of generality, we can restrict ourselves to roof vectors from the  simplex $\Delta^{m-1} = \{\vec{s}\in \R^m_+:\ \sum_{j=1}^m s_j =1\}$.

\begin{theorem} \label{th-prod}
Let $\zeta$ be a primitive aperiodic substitution on $\Ak=\{1,\ldots,m\}$, with substitution matrix $\Sf$.  Suppose that the characteristic polynomial of $\Sf$ is irreducible and the second eigenvalue  satisfies $|\theta_2|>1$. Then
 there exists a constant $\alpha>0$, depending only on the substitution $\zeta$, such that for Lebesgue-almost every $\vec{s}\in  \Delta^{m-1}$, for every   Lip-cylindrical function $f$, with $\int f\,d\wtil{\mu} =0$, 
 and for any  ergodic flow $(Y, H_t,\nu)$ and a function $g\in L^2_0(Y,\nu)$, 
 \be\label{prodest1}
\left| \int_0^R \langle (f\otimes g)\circ(h_t\times H_t), f\otimes g\rangle\,dt\right| \le CR^{1-\alpha},\ \ R>0,
 \ee
where $C=C(\vec{s},\|f\|_L, \|g\|_2)>0$.
\end{theorem}

In order to prove Theorem~\ref{th-prod}, we need the following strengthening of Theorem 4.1 in \cite{BuSo1}.

\begin{theorem} \label{th-holder1} Let $\zeta$ be a primitive aperiodic substitution on $\Ak$, satisfying the assumptions of Theorem~\ref{th-prod}. Then
 there exists a constant $\gam>0$, depending only on the substitution $\zeta$, such that for Lebesgue-almost every $\vec{s}\in  \Delta^{m-1}$ there exists $r_0=r_0(\vec{s})>0$, such that for every Lip-cylindrical function $f$, with $\int f\,d\wtil{\mu} =0$, 
\be \label{holder1}
\sig_f ([\om-r,\om+r]) \le C r^\gam,\ \ \mbox{for all}\ \om\in \R\ \mbox{and}\  0 < r \le r_0.
\ee
Here $\sig_f$ is the spectral measure of $f$ corresponding to the suspension flow $(\Xxi^{\vec{s}},h_t)$ and $C>0$ depends only on $\|f\|_L$.
\end{theorem}

The improvement upon theorem 4.1 in \cite{BuSo1} is that in (\ref{holder1}) our local H\"older estimates are {\bf  uniform} on the {\bf whole 
line}, while in \cite{BuSo1} we were only able to prove  our estimates to be uniform away from zero and infinity.
An estimate of the Hausdorff dimension of the exceptional set of suspension flows can also be given, cf. \cite[Theorem 4.2]{BuSo1}.

\section{ Proof of Theorem~\ref{th-prod} assuming Theorem~\ref{th-holder1}.}
By a result of Strichartz \cite[Cor.\,5.2]{Str90} we immediately obtain from (\ref{holder1}):
\be \label{stri}
\sup_y \sup_{R\ge 1} R^{\gam-1} \int_{y-R}^{y+R} |\widehat{\sig}_f(\zeta)|^2\,d\zeta \le C.
\ee
By the definition of spectral measures,
\begin{eqnarray*}
\int_0^R \widehat{\sigma}_f(-\zeta)\widehat{\sigma}_g(-\zeta)\,d\zeta & = & \int_0^R \left(\int_{\Xxi^{\vec{s}}} (f\circ h_t)\,\ov{f}\,d\wtil{\mu}\right)\left(\int_Y (g\circ H_t)\,\ov{g}\,d\nu\right)\\
& = & \int_0^R \int_{\Xxi^{\vec{s}}\times Y} \bigl((f\otimes g)\circ (h_t \times H_t)\bigr)\,(\ov{f\otimes g})\,d(\widetilde{\mu}\times \nu),
\end{eqnarray*}
which is exactly the expression in (\ref{prodest1}) under the absolute value sign. It remains to note that 
$$
\left| \int_0^R \widehat{\sigma}_f(-\zeta)\widehat{\sigma}_g(-\zeta)\,d\zeta \right| \le \left(\int_0^R |\widehat{\sig}_f(-\zeta)|^2\,d\zeta\right)^{1/2} \,\left(\int_0^R |\widehat{\sig}_g(-\zeta)|^2\,d\zeta\right)^{1/2} \le C^{1/2}R^{1-\frac{\gam}{2}}\|g\|_2,
$$
applying Cauchy--Buniakovsky--Schwarz, (\ref{stri}), and the simple bound $\|\widehat{\sigma}_g\|_\infty \le \|g\|^2_2$.
\qed

\bigskip

The plan of the proof of Theorem~\ref{th-holder1} is as follows: we go through the proof of
\cite[Theorem 4.2]{BuSo1}, making it more quantitative, and obtain 

\begin{prop} \label{prop-holder2}
Let $\zeta$ be a primitive aperiodic substitution on $\Ak$, satisfying the assumptions of Theorem~\ref{th-prod}. Then there exist constants $\wtil{\gam}, Z>0$, depending only on the substitution $\zeta$, such that for Lebesgue-almost every $\vec{s}\in  \Delta^{m-1}$  there exists $r_0=r_0(\vec{s})$, such that 
 for every Lip-cylindrical function $f$ and $\om\ne 0$,
\begin{equation} \label{hole1}
\sig_f([\om-r,\om+r]) \le C\cdot r^{\wtil{\gam}},\ \ \mbox{for}\ \  0 < r < r_0 |\om|^Z.
\end{equation} 
Here the constant $C=C(\|f\|_L)>0$ depends only on the Lip-norm of $f$.
\end{prop}

Note that here we do not have to assume $\int_X f\,d\wtil{\mu}=0$.
We will then ``glue'' this H\"older bound with the H\"older bound at $\om=0$ (which essentially follows from a result of Adamczewski \cite{Adam})
in the case when $f$ has mean zero.


\section{Twisted ergodic integrals and spectral measures}

 Let $(Y,\mu,h_y)$ be a measure-preserving flow. For $f\in L^2(Y,\mu),\ R>0$, $\om\in \R$, and $y\in Y$ consider the ``twisted Birkhoff integral''
$$
S_R^{(y)}(f,\om) = \int_0^R e^{-2\pi i \om t} f(h_t y)\,dt.
$$
Recall the following standard lemma; a proof may be found in \cite[Lemma 4.3]{BuSo1}. 

\begin{lemma} \label{lem-varr}
 Let $\Om(r)$ be a continuous increasing function on $[0,1)$, such that $\Om(0)=0$, and suppose that for some fixed $\om \in \R$, $R_0\ge 1$, $C_1>0$, and $f\in L^2(Y,\mu)$ we have
\be \label{eq-SR}
\sup_{y\in Y} |S^{(y)}_R(f,\om)| \le R \sqrt{C_1\Om(1/R)}\ \ \mbox{for}\ R\ge R_0,
\ee
Then
\be \label{eq-Hof12}
\sig_f([\om-r,\om+r]) \le \frac{\pi^2 C_1}{4} \, \Om(2r) \ \ \mbox{for all}\ r \le (2R_0)^{-1}.
\ee
\end{lemma}


Recall that our test functions depend only on the cylinder set $X_a$ and the height $t$. More precisely, given some functions
 $\psi_a\in \Lip([0,s_a])$, $a\in \A$, let
\be \label{fcyl}
f = \sum_{a\in \A} f_a,\ \ \ \mbox{with}\ \ f_a(x,t) = \One_{\Xx_a} \psi_a(t),\ \ \mbox{where}\ \ 
\Xx_a=X_a\times [0,s_a].
\ee
For a word $v$ in the alphabet $\Ak$ denote by $\vec{\ell}(v)\in \Z^m$ its ``population vector'' whose $j$-th entry is the number of $j$'s in $v$, for $j\le m$. We will  need the
``tiling length'' of $v$ defined by 
\be \label{tilength}
|v|_{\vec{s}}:= \langle\vec{\ell}(v), \vec{s}\rangle.
\ee

\noindent For $v=v_0\ldots v_{N-1}\in \Ak^+$ let
\be \label{def-Phi3}
\Phi_a^{\vec{s}}(v,\om) = \sum_{j=0}^{N-1} \delta_{v_j,a} \exp(-2\pi i \om |v_0\ldots v_j|_{\vec{s}}).
\ee
Then a straightforward calculation shows
\be \label{SR1}
S_R^{(x,0)}(f_a,\om) = \widehat{\psi}_a(\om) \cdot {\Phi_a^{\vec{s}}(x[0,N-1],\om)}\ \ \ \mbox{for}\ \ R = \left|x[0,N-1]\right|_{\vec{s}}.
\ee

Next we quote Proposition 4.4 from \cite{BuSo1}, with a tiny modification. The symbol $\|x\|$ denotes the distance from $x\in \R$ to the nearest integer (when we use $\|\cdot\|$ for a norm, this is always indicated by a subscript).

\begin{prop}  \label{prop-Dioph0}
Let $\zeta$ be a primitive substitution on $\A$ and $v$ a  word (called ``return word'') starting with $c\in\A$, such that $vc$ occurs as a subword in  $\zeta(b)$ for every $b\in \A$. Let $\vec{s} \in \Delta^{m-1}$.
Then there exist $c_1\in (0,1)$ and $C,C',C_2>0$, depending only on the substitution $\zeta$ and $\min_j s_j$, such that 

{\bf (i)} for all $a,b\in \A$, $n\in \Nat$, and $\om\in \R$,
\be \label{eq-new19}
|\Phi_a^{\vec{s}}(\zeta^n(b),\om)|\le C |\zeta^n(b)|_{\vec{s}} \cdot \prod_{k=0}^{n-1} \Bigl(1 - c_1\bigl\|\om\,|\zeta^k(v)|_{\vec{s}}\bigr\|^2\Bigr);
\ee

{\bf (ii)} for all $R>1,\ \om \in \R$, and a cylindrical Lipschitz function $f$,
\be \label{eq-matrest10}
|S^{(x,t)}_R(f,\om)| \le C' \|f\|_\infty \cdot \min\{1,|\om|^{-1}\}\cdot R\ \cdot\!\!\!\!\prod_{k=0}^{\lfloor \log_\theta R- C_2\rfloor} (1 - c_1\bigl\|\om |\zeta^k(v)|_{\vec{s}}\bigr\|^2)\ \ \ \mbox{for all}\ \ (x,t) \in \Xxi^{\vec{s}},
\ee
where $\theta$ is the Perron-Frobenius eigenvalue of the substitution matrix $\Sf=\Sf_\zeta$.
\end{prop}

The only difference from \cite{BuSo1} is that there we only considered characteristic functions of cylinder sets, instead of general cylindrical functions. However, the proof is exactly the same, taking (\ref{SR1}) into account, and the well-known inequality for the Fourier transform of a Lipschitz function: 
$$|\widehat{\psi}_a(\om)| \le \const\cdot \|\psi_a\|_L \cdot\min\{1, |\om|^{-1}\}\le  \const\cdot \|f\|_L \cdot\min\{1, |\om|^{-1}\}.
$$


\section{Proof of Proposition~\ref{prop-holder2}}
\subsection{Preliminary considerations}
The proof relies on the so-called ``Erd\H{o}s-Kahane argument'', which originated in the study of infinite Bernoulli convolutions, see \cite{Erd,Kahane}.
While the proof of Proposition~\ref{prop-holder2} follows the general scheme of that of \cite[Theorem 4.2]{BuSo1}, the technical  implementation of the Erd\H{o}s-Kahane argument is quite different, see Proposition \ref{prop-EKvar} below.

Recall that, passing to a power $\zeta^\ell$ if necessary, we can always obtain a return word $v$ as in the statement of Proposition~\ref{prop-Dioph0}, and the existence of such a word (for $\zeta$ itself)  will be the standing assumption until the end of the section.


Let $\theta_1=\theta, \theta_2,\ldots,\theta_m$ be the eigenvalues of the  substitution matrix $\Sf$, ordered by magnitude, and let $\vec{e}_j$ be the corresponding eigenvectors of unit norm (real and complex). (Recall that  irreducibility of the characteristic polynomial of  $\Sf$ implies diagonalizability over $\C$.) Suppose that $\Sf$ has exactly $q$ eigenvalues of absolute value $\le 1$, for some $q< m-1$. In other words,
$$
|\theta_{m-q}|>1,\ \ \ |\theta_{m-q+1}|\le 1
$$
(we do not exclude the possibility of $q=0$; in that case the second inequality is vacuous).
Let $\{\vec{e_j^*}\}_1^m$ be the dual basis, i.e.\ $\vec{e_j^*}$ is the eigenvector of the transpose $\Sf^t$ corresponding to $\theta_j$, such that $\langle \vec{e}_i, \vec{e_j^*}\rangle = \delta_{ij}$. 
Then $\vec{s} = \sum_{j=1}^m \langle \vec{e}_j,\vec{s}\rangle \vec{e_j^*}$, hence
$$
|\zeta^n(v)|_{\vec{s}}=\langle \vec{\ell}(\zeta^n(v)), \vec{s}\rangle = \langle \Sf^n \vec{\ell}(v),\vec{s}\rangle = \sum_{j=1}^m \langle \vec{e}_j,\vec{s}\rangle \, \langle \vec{\ell}(v),\vec{e_j^*}\rangle \,\theta_j^n,\  \ n\ge 0.
$$
Let
\be \label{coord}
b_j = \langle \vec{e}_j,\vec{s}\rangle \, \langle \vec{\ell}(v),\vec{e_j^*}\rangle,\ j=1,\ldots,m,
\ee
so that
\be \label{coord2}
|\zeta^n(v)|_{\vec{s}}= \sum_{j=1}^m b_j \theta_j^n.
\ee
We always have $b_1>0$, since $\theta_1$ is the Perron-Frobenius eigenvalue, both eigenvectors $\vec{e}_1$ and $\vec{e_1^*}$ are strictly positive, $\vec{s}$ is strictly positive, and $\vec{\ell}(v)\ne \vec{0}$ is
non-negative. Further, since $\vec{\ell}(v)$ is an integer vector and the characteristic polynomial of $\Sf$ is irreducible, we have $\langle \vec{\ell}(v),\vec{e_j^*}\rangle \ne 0$ for all $j\le m$. Indeed, otherwise $\Sf$ would have a rational invariant subspace, spanned by $\Sf^n \vec{\ell}(v),\ n\ge 0$, of dimension less than $m$, contradicting the fact that its eigenvalues are algebraic integers of degree $m$. Note also that
$b_{j'}=\ov{b_j}$ for $\theta_{j'} = \ov{\theta_j}$.
Let
$$
\Hk^{m-1} = \{(a_1,\ldots,a_m)\in \C^{m}:\ a_1=1,\ a_{j'} = \ov{a_j}\ \mbox{for}\ \theta_{j'} = \ov{\theta_j}\},
$$
in particular, $a_j$ are real for real eigenvalues $\theta_j$. Further,
let $P_{m-q}$ be the projection from $\Hk^{m-1}$ to the subspace spanned by the first $m-q$ coordinates, and let $\Hk^{m-q-1} = P_{m-q}\Hk^{m-1}$. It is clear that $\Hk^{m-1}$ is a real affine-linear space of dimension $m-1$ and $\Hk^{m-q-1}$ is a real affine-linear space
of dimension $m-q-1$.  It is convenient to pass from $\Delta^{m-1}$ to a subset of $\Hk^{m-1}$ when parametrizing the suspension flows. To this end, consider the map ${\Fk}:\,\Delta^{m-1} \to \C^{m}$ given by
\be \label{def-Fk}
\Fk(\vec{s}) = \left(\frac{\langle \vec{e}_j,\vec{s}\rangle \, \langle \vec{\ell}(v),\vec{e_j^*}\rangle}{\langle \vec{e}_1,\vec{s}\rangle \, \langle \vec{\ell}(v),\vec{e_1^*}\rangle}\right)_{1\le j \le m}.
\ee
The map ${\Fk}$ is a change of basis transformation, which it is linear and invertible, followed by 
division by the first coordinate. Notice that $\langle \vec{e}_1,\vec{s}\rangle \, \langle \vec{\ell}(v),\vec{e_1^*}\rangle$ is positive and bounded away from zero by a constant depending only on $\zeta$ and on $v$ (and since $v$ is fixed, it depends only on $\zeta$). Note also that $\Fk(\Delta^{m-1}) \subset \Hk^{m-1}$. Thus $\Fk$ is 1-to-1 and $\|\Fk^{-1}\|_\infty$ depends only on $\zeta$, where $\Fk^{-1}$ is considered on the range $\Fk(\Delta^{m-1})$. It is also
clear that $\Fk$ preserves  Hausdorff dimension.

\subsection{A variant of the Erd\H{o}s-Kahane argument}
The following proposition contains the core of the proof of Proposition~\ref{prop-holder2}, and it is different in many technical details from the corresponding Proposition 4.5 in \cite{BuSo1}. Consider the Vandermonde matrix
\be \label{Vand}
\Theta = \left( \begin{array}{ccc} 1 & \ldots & 1 \\ \vdots & \ddots & \vdots \\  \theta_1^{m-1} & \ldots & \theta_m^{m-1} \end{array} \right)
\ee
and the $\ell^\infty$ operator norms $\|\Theta^{\pm 1} \|_\infty$; note that $\Theta$ is invertible, since all $\theta_j$ are distinct. 

\begin{prop} \label{prop-EKvar} Let  $k\in \Nat$. Consider two constants, depending only on the substitution matrix (actually, on
$\Theta$), defined as follows:
\be \label{def-Lrho}
\rho:= \half (1+\theta_1\|\Theta\|_\infty\|\Theta^{-1}\|_\infty)^{-1}\ \ \ \mbox{and}\ \ \ L:=2+\theta_1\|\Theta\|_\infty\|\Theta^{-1}\|_\infty.
\ee
For $B\ge 2$  let $E_k^N(B)$ be the set of
$(a_1,\ldots,a_{m-q})\in \Hk^{m-q-1}$  such that  there exist $\om\in [B^{-1},B]$ and $a_{m-q+1},\ldots,a_m$, with $(a_1,\ldots,a_m)\in \Fk(\Delta^{m-1})$, for which
\be \label{eq2}
 \card \Bigl\{n\in [1,N]:\ \Bigl\|\om\sum_{j=1}^m a_j \theta_j^n\Bigr\|\ge \rho\Bigr\} < \frac{N}{k}\,.
\ee
Further, for $\Ups>0$ let 
$$E_k(\Ups):=\bigcap_{N_0=1}^\infty \ \ \bigcup_{B=2}^\infty \ \ \bigcup_{N=N_0+\lfloor \Ups \log B \rfloor}^\infty E_k^N(B).$$
Then 
$$\lim_{\Ups\to \infty} \lim_{k\to \infty} \dim_H (E_k(\Ups))=0.$$
\end{prop}

\begin{proof}
For $\om >0$ and $(a_1,\ldots,a_m)\in \Fk(\Delta^{m-1})$, let
\be \label{eq3}
\om \sum_{j=1}^m a_j \theta_j^n = K_n +\eps_n,\ \ K_n \in \Z,\ \ |\eps_n|\le 1/2,\ n\ge 1,
\ee
so that $\|\om \sum_{j=1}^m a_j \theta_j^n \|=|\eps_n|$.
Denote
$$
\vec{a} = \left( \begin{array}{c} a_1 \\ \vdots \\ a_m \end{array}\right),\ \ \ \vec{K}_n = \left( \begin{array}{c} K_n \\ \vdots \\ K_{n+m-1} \end{array}\right),\ \ \mbox{and}\ \  
\vec{\eps}_n = \left( \begin{array}{c} \eps_n \\ \vdots \\ \eps_{n+m-1} \end{array}\right);
$$
then  equations (\ref{eq3}) for $n, n+1,\ldots,n+m-1$ combine into 
\be \label{matr1}
\om \left( \begin{array}{ccc} \theta_1^n & \ldots & \theta_m^n \\ \vdots & \ddots & \vdots \\ \theta_1^{n+m-1} & \ldots & \theta_m^{n+m-1} \end{array} \right) \vec{a} = \vec{K}_n + \vec{\eps}_n.
\ee
Let ${\rm Diag}[\theta_j^n]$ be the diagonal matrix with the diagonal entries $\theta_1^n,\ldots,\theta_m^n$, 
then (\ref{matr1}) becomes
$$
\om \,\Theta \cdot {\rm Diag}[\theta_j^n] \,\vec{a} = \vec{K}_n + \vec{\eps}_n,\ \ \ n\ge 1,
$$
where $\Theta$ is the Vandermonde matrix (\ref{Vand}).
The Vandermonde matrix is invertible, since $\theta_j$ are all distinct. Also, all $\theta_j$ are nonzero since $\Sf$ is irreducible, hence
\be \label{eq39}
\vec{a} = \om^{-1} {\rm Diag}[\theta_j^{-n}]\,\Theta^{-1}(\vec{K}_n + \vec{\eps}_n),\ \ \ n\ge 1.
\ee
Now, comparing (\ref{eq39}) with the same equality for $n+1$, we obtain
\be \label{ura}
\vec{K}_{n+1} + \vec{\eps}_{n+1} = \Theta \,{\rm Diag}[\theta_j]\,\Theta^{-1} (\vec{K}_n + \vec{\eps}_n).
\ee

\begin{lemma} [Lemma 4.6 in \cite{BuSo1}]\label{lem-step} 
Let $\rho$ and $L$ be the constants given by (\ref{def-Lrho}). Consider arbitrary $\om>0$ and $\vec{a} = (a_1,\ldots,a_m) \in \Hk^{m-1}$, and define $K_n,\eps_n$, $n\ge 1$, by the formula (\ref{eq3}).

{\bf (i)} if\ \ $\max\{|\eps_n|,\ldots,|\eps_{n+m}|\} <  \rho$, then $K_{n+m}$ is uniquely determined by
$K_n, K_{n+1}, \ldots, K_{n+m-1}$, independent of $\om$ and $(a_1,\ldots,a_m)$;

{\bf (ii)} given $K_n, K_{n+1}, \ldots, K_{n+m-1}$, there are at most $L$ possibilities for $K_{n+m}$.
\end{lemma}

We proceed with the proof of Proposition \ref{prop-EKvar}.
It follows from (\ref{eq39}) that
\be \label{eq4}
a_j =\om^{-1} \theta_j^{-n} [\Theta^{-1}(\vec{K}_n + \vec{\eps}_n)]_{_{\scriptstyle j}},\ \ \ j=1,\ldots,m,\ \ n\ge 1,
\ee
where $[\cdot]_j$ denotes the $j$-th component of a vector. 
Assuming that $(a_1,\ldots,a_{m-q}) \in E_k^N(B)$ we will have $\om\in [B^{-1},B]$ and $a_1=1$, hence 
\be \label{kapusta0}
B^{-1} \theta_1^n \le |[\Theta^{-1}(\vec{K}_n + \vec{\eps}_n)]_{_{\scriptstyle 1}}| \le B \theta_1^n,\ \ n\ge 1.
\ee
Furthermore,  $\Fk^{-1}(a_1,\ldots,a_m) \in \Delta^{m-1}$ implies $|a_j| \le \|\Fk^{-1}\|_\infty$, therefore,
\be \label{kapusta}
 \left|[\Theta^{-1}(\vec{K}_n + \vec{\eps}_n)]_{_{\scriptstyle j}}\right| \le  \Bu\|\Fk^{-1}\|_\infty \cdot |\theta_j|^n,\ \ \ j\le m-q, \ \ n\ge 1.
\ee
From (\ref{eq4}), recalling that $a_1=1$, we obtain 
\be \label{eq-a}
a_j = \frac{\theta_j^{-n} [\Theta^{-1}(\vec{K}_n + \vec{\eps}_n)]_{_{\scriptstyle j}}}{\theta_1^{-n} [\Theta^{-1}(\vec{K}_n + \vec{\eps}_n)]_{_{\scriptstyle 1}}}\approx \frac{\theta_j^{-n} [\Theta^{-1}\vec{K}_n ]_{_{\scriptscriptstyle j}}}{\theta_1^{-n} [\Theta^{-1}\vec{K}_n ]_{_{\scriptscriptstyle 1}}}\,\ \ \ j\le m-q,
\ee
for $n$ sufficiently large, and we need to be precise about this. We certainly want $\vec{K}_n$ to be a positive vector, which in view of (\ref{eq3}), is guaranteed when $n\ge O_\zeta(1)\cdot \log B$ (here and below we denote by $O_\zeta(1)$ a constant which depends only on the substitution $\zeta$).
To estimate the error in the approximation above, we can write
\be\label{kapusta2}
\left| \frac{ [\Theta^{-1}(\vec{K}_n + \vec{\eps}_n)]_{_{\scriptstyle j}}}{[\Theta^{-1}(\vec{K}_n + \vec{\eps}_n)]_{_{\scriptstyle 1}}} - 
\frac{ [\Theta^{-1}\vec{K}_n ]_{_{\scriptstyle j}}}{ [\Theta^{-1}\vec{K}_n ]_{_{\scriptstyle 1}}}\right| \le  
\frac{| [\Theta^{-1} \vec{\eps}_n]_{_{\scriptstyle j}}|}{|[\Theta^{-1}(\vec{K}_n + \vec{\eps}_n)]_{_{\scriptstyle 1}}|} + 
\frac{| [\Theta^{-1} \vec{\eps}_n]_{_{\scriptstyle 1}} [\Theta^{-1}\vec{K}_n ]_{_{\scriptstyle j}}|}{|[\Theta^{-1}(\vec{K}_n + \vec{\eps}_n)]_{_{\scriptstyle 1}} [\Theta^{-1}\vec{K}_n ]_{_{\scriptstyle 1}}|}\,.
\ee
Observe that 
$$\|\Theta^{-1}\vec{\eps}_n\|_\infty \le \|\Theta^{-1}\|_\infty \|\vec{\eps}_n\|_\infty\le 
(1/2) \|\Theta^{-1}\|_\infty=:C_\Theta.$$ 
Thus we can continue (\ref{kapusta2}) to obtain for $j\le m-q$:
$$
\mbox{(\ref{kapusta2})} \ \le\  \frac{C_\Theta}{B^{-1}\theta_1^n} \left( 1 + \frac{B\|\Fk^{-1}\|_\infty |\theta_j|^n + C_\Theta}{B^{-1}\theta_1^n - C_\Theta}\right) \le 2B C_\Theta \theta_1^{-n}
$$
for $n\ge O_\zeta(1)\cdot \log B$, since $|\theta_j|<\theta_1$.  (The constant in the lower bound for $n$ depends only on the substitution, since $C_\Theta$ and $\|\Fk^{-1}\|_\infty$ are determined by $\zeta$.) Therefore, by the equality in (\ref{eq-a}),
\be\label{eq5}
\left|a_j - \frac{\theta_j^{-n} [\Theta^{-1}\vec{K}_n ]_{_{\scriptstyle j}}}{\theta_1^{-n} [\Theta^{-1}\vec{K}_n ]_{_{\scriptstyle 1}}}\right| \le 2B C_\Theta \cdot |\theta_j|^{-n},\ \ \ j\in\{2,\ldots,m-q\},\ \ n\ge O_\zeta(1)\cdot\log B.
\ee
It is crucial, of course, that $|\theta_j|>1$ for $j\in\{2,\ldots,m-q\}$.

We conclude the proof of Proposition~\ref{prop-EKvar}. 
We estimate the Hausdorff dimension of $E_k(\Ups)$ from above by producing efficient covers of $E_k^N(B)$. Consider an arbitrary point 
$$(a_1,\ldots,a_{m-q}) \in E^N_k(B), \ \ N\ge N_0 + \lfloor \Ups \log B\rfloor.$$ By definition, we can find $\om\in[B^{-1},B]$ and $a_{m-q+1},\ldots,a_m$, with $(a_1,\ldots,a_m)\in \Fk(\Delta^{m-1})$, for which (\ref{eq2}) holds. We then
find the numbers $K_n,\eps_n$ from (\ref{eq3}). The inequality (\ref{eq5}) was proved for $n\ge  O_\zeta(1)\cdot \log B$, and we can apply it for $n = N-m+1$, assuming that
$\Ups > O_\zeta(1)$.
Using that  $$|\theta_{m-q}| = \min_{j\le m-q}|\theta_j|>1,$$ we obtain that $(a_1,\ldots,a_{m-q})$ is contained in the closed $\ell^\infty$ ball of
 radius $2BC_\Theta\cdot |\theta_{m-q}|^{-N+m-1}$, centered at the point
 $$
 (x_1,\ldots,x_{m-q}),\  \mbox{where}\  x_1=1\ \mbox{and}\ x_j = \frac{\theta_j^{-N+m-1} [\Theta^{-1}\vec{K}_{N-m+1} ]_{_{\scriptstyle j}}}{\theta_1^{-N+m-1} [\Theta^{-1}\vec{K}_{N-m+1} ]_{_{\scriptstyle 1}}},\ \ j=2,\ldots,m-q.
 $$
The number of such balls does not exceed the number of possible vectors $\vec{K}_{N-m+1}$. This, in turn, is bounded above by the
number of possible sequences $K_1,\ldots,K_N$. 
Now we use the crucial assumption (\ref{eq2}) in the definition of the set $E_k^N(B)$. The set $\{n\in [1,N]:\ |\eps_n| \ge \rho\}$                has cardinality less than $N/k$, and we can enlarge it arbitrarily
to get a set $\Gam \subset [1,N]\cap \Nat$ with $\card(\Gam) = \lceil \frac{N}{k} \rceil$. There are  ${N \choose \lceil N/k \rceil}$ such subsets $\Gam$, and it remains to estimate
the number of possible sequences $K_1,\ldots,K_N$ for a fixed $\Gam$.

Recall that
$
|a_j| \le \|\Fk^{-1}\|_\infty, \ j=1,\ldots,m.
$
Further, $|\om| \in [B^{-1},B]$, hence  (\ref{eq3}) implies an upper bound 
$$
|K_n| \le O_\zeta(1)\cdot B,\ \ n\ge 1.
$$
Thus, there are no more than $O_\zeta(1)\cdot B^m$ possibilities for the number of initial sequences $K_1,\ldots,K_m$.

Now we fix $\Gam\subset [1,N]\cap \Nat$ and consider those
$(a_1,\ldots,a_{m-q})$ for which
$|\eps_n|< \rho$ for $n\in [1,N]\setminus \Gamma$. 
Once $K_1,\ldots,K_n$ are determined, for $m\le n \le N-1$, we check whether $\{n-m+1,\ldots,n+1\}$ intersects $\Gam$. If it does, there are at most $L$ possibilities for $K_{n+1}$ by Lemma~\ref{lem-step}(ii). If it does not, then there is only one choice of $K_{n+1}$. It follows that the number of sequences $K_1,\ldots,K_N$ for the given $\Gam$ does not exceed $O_\zeta(1)\cdot B^m\cdot L^{(m+1)\card(\Gam)}$. 

Thus, the total number sequences, hence the balls of radius $2BC_\Theta\cdot |\theta_{m-q}|^{-N+m-1}$ needed to cover $E_k^N(B)$ is at most
$$
O_\zeta(1)\cdot B^m \cdot {N \choose \lceil N/k \rceil} \cdot L^{(m+1)\lceil N/k\rceil }.
$$ 
Therefore, we can estimate the Hausdorff measure $\Hk^\eta(E_k(\Ups))$, for a fixed $\eta\in (0,1)$, as follows: for all $N_0\ge 1$,
$$
\Hk^\eta(E_k(\Ups)) \le O_\zeta(1) \cdot \sum_{B=2}^\infty B^m \sum_{N= N_0 + \lfloor \Ups \log B\rfloor} {N \choose \lceil N/k \rceil} \cdot L^{(m+1)\lceil N/k\rceil }\left(2BC_\Theta\cdot |\theta_{m-q}|^{-N+m-1}\right)^\eta
$$
Stirling's formula implies that ${N \choose \lceil N/k \rceil} \le \exp[\wtil{C}(k^{-1}\log k)N]$ for some $\wtil{C}>0$, so we obtain from the above:
$$
\Hk^\eta(E_k(\Ups)) \le O_\zeta(1) \cdot \sum_{B=2}^\infty B^{m+1}\sum_{N= N_0 + \lfloor \Ups \log B\rfloor} \exp\Bigl[\Bigl(\frac{\wtil{C}\log k}{k} + \frac{m \log L}{k} - \eta \log|\theta_{m-q}|\Bigr)N \Bigr] 
$$
Now, choosing $k$ sufficiently large, in such a way that 
$
k^{-1}({\wtil{C}\log k} + {m \log L})< \frac{\eta}{2}\log|\theta_{m-q}|\,,
$
we obtain
\begin{eqnarray*}
\Hk^\eta(E_k(\Ups)) & \le & O_\zeta(1) \cdot \sum_{B=2}^\infty B^{m+1} \sum_{N= N_0 + \lfloor \Ups \log B\rfloor} \exp(- N\eta \log|\theta_{m-q}|/2) \\
& \le & O_\zeta(1) \cdot \sum_{B=2}^\infty B^{m+1} \cdot B^{-\Ups\cdot \eta \log|\theta_{m-q}|/2}\exp(- N_0\eta \log|\theta_{m-q}|/2).
\end{eqnarray*}
Choosing $\Ups$ sufficiently large, in such a way that $\Ups\eta \log|\theta_{m-q}|/2 > m+2$, we obtain a convergent series in $B$, and since the inequality holds for any $N_0$, we will get 
$\Hk^\eta(E_k(\Ups))=0$ for the appropriate $k$ and $\Ups$.
The proof of Proposition~\ref{prop-EKvar} is complete.
\end{proof}

\subsection{Proof of Proposition~\ref{prop-holder2}}
Choose $k\in \Nat$ and $\Ups>0$ in such a way that $\dim_H(E_k(\Ups)) <1$, which is possible by Proposition~\ref{prop-EKvar}.
Let
$$
\Ek_k(\Ups):= \Fk^{-1} P_{m-q}^{-1}(E_k(\Ups)).
$$
Note that $P_{m-q}^{-1}(E_k(\Ups))$ is the direct product of $E_k(\Ups)$ with a real $q$-dimensional linear space, hence
$\dim_H(\Ek_k(\Ups))=\dim_H(E_k(\Ups))+q<1+q$.  We want to show that $\Ek_k(\Ups)$ is the desired exceptional set in Proposition~\ref{prop-holder2}. To this end, let $\vec{s}\in \Delta^{m-1}\setminus \Ek_k(\Ups)$. Consider the coefficients $b_j$ defined by (\ref{coord}), so that (\ref{coord2}) holds; then $$\Fk(\vec{s}) = (1, b_2/b_1,\ldots, b_m/b_1)=:(a_1,\ldots,a_{m}).$$
Observe that
$$
b_1 = \langle \vec{e}_1,\vec{s}\rangle \, \langle \vec{\ell}(v),\vec{e_1^*}\rangle\in [C_3^{-1},C_3],
$$
where $C_3>1$ depends only on $\zeta$ and $v$, since 
$$
\min_j(\vec{e}_1)_j\le \langle \vec{e}_1,\vec{s}\rangle\le \max_j(\vec{e}_1)_j\ \ \mbox{for all}\ \vec{s} \in \Delta^{m-1}.
$$
Let $\om\ne 0$. By symmetry, we can assume that $\om>0$. Let $B\ge 2$ be minimal such that $[C_3^{-1}\om, C_3 \om] \subset [B^{-1}, B]$, that is,
\be \label{def-B}
B = \lceil C_3 \max(\om,\om^{-1})\rceil.
\ee
From $\vec{s}\not\in \Ek_k(\Ups)$ it follows that 
$$(a_1,\ldots,a_{m-q}) \not\in E_k(\Ups),$$ hence there exists $N_0=N_0(\vec{s})\in \Nat$ such that
$$(a_1,\ldots,a_{m-q}) \not \in E_k^N(B),\ \ \mbox{for all}\ N\ge N_0+\lfloor \Ups\log B\rfloor.
$$
By the definition of $E_k^N(B)$ and (\ref{coord2}), rescaling by $b_1\in [C_3^{-1},C_3]$, we obtain that there are at least $\lfloor N/k\rfloor$ integers $n\in [1,N]$ for 
which $$\|\om|\zeta^n(v)|_{\vec{s}}\|\ge \rho,$$ hence
$$
 \prod_{n=1}^{N} (1-c_1\|\om|\zeta^n(v)|_{\vec{s}}\|^2)\le (1-c_1\rho^2)^{\lfloor N/k\rfloor},\ \ \mbox{for all}\ N\ge N_0 + \lfloor \Ups\log B\rfloor.
$$
Combined with Proposition~\ref{prop-Dioph0}(ii), this estimate implies, for all $\om>0$: 
\be \label{imp}
\sup\left\{ \bigl|S_R^{(x,t)}(f,\om)\bigr|:\ (x,t)\in\Xxi^{\vec{s}}\right\}\le C'\|f\|_L\cdot R\cdot(1-c_1\rho^2)^{\lfloor(\log_\theta R - C_2)/k\rfloor} \le C''\|f\|_L\cdot R^\alpha,
\ee
where
$$
\alpha = 1+\frac{\log_\theta(1-c_1 \rho^2)}{k}\in (0,1),
$$
as long as $\log_\theta R- C_2 > N_0 + \Ups \log B$, 
 which can be written as $$R > C_4 \cdot B^{\Ups \log \theta} \ge C_4'\cdot \max\{|\om|^Z, |\om|^{-Z}\},\ \ 
Z = \Ups\log\theta,$$
for some constants $C_4, C_4'$ depending only on $\zeta$ and $\vec{s}$,
in view of (\ref{def-B}). When $|\om|>1$ and $R < C_4' |\om|^Z$, we simply ignore the product term in (\ref{eq-matrest10}) and write
$$
\sup\left\{ \bigl|S_R^{(x,t)}(f,\om)\bigr|:\ (x,t)\in\Xxi^{\vec{s}}\right\}\le C'\|f\|_L\cdot R\cdot |\om|^{-1} < C'{C'_4}^{1/Z}\|f\|_L\cdot R^{1-1/Z}.
$$
Now the claim of Proposition~\ref{prop-holder2}, with $\wtil{\gam} = 2-2\beta$, where $\beta = \max\{\alpha, 1-1/Z\}$, follows from Lemma~\ref{lem-varr}, with $\Omega(r)=r^{\wtil{\gam}}$.\qed


\section{Proof of Theorem~\ref{th-holder1}}

In order to prove Theorem~\ref{th-holder1}, we need an estimate of the spectral measure at zero. 

\begin{prop} \label{prop-zero}
Let $\zeta$ be a primitive substitution, with the second eigenvalue of the substitution matrix equal to $\theta_2$,
with $|\theta_2|>1$,
and the total number of eigenvalues equal to $|\theta_2|$ in absolute value is $\alpha+1$. For any $\vec{s}\in \Delta^{m-1}$ let $f$ by a cylindrical function on $\Xxi^{\vec{s}}$ such that
$\int f\,d\wtil{\mu}=0$ and $\sig_f$ the corresponding spectral measure. Then
$$
\sig_f([-r,r]) = O_{\zeta,\vec{s}} (1) \cdot (\log(1/r))^{2\alpha} r^{2-2\beta},\ \ r>0,\ \ \mbox{where}\ \ \beta = \log_\theta|\theta_2|.
$$
\end{prop}

This proposition is a consequence of a result, essentially due to Adamczewski, on the symbolic discrepancy for substitutions. It is stated in the context of $\Z$-actions. Let 
$F = \sum_{a\in \Ak} d_a \One_{[a]}$ and consider the Birkhoff sum
\be \label{Birsum}
S_N^x(F) := \sum_{n=0}^{N-1} F(T_\zeta^n x) = \sum_{n=0}^{N-1} d_{x_n}, \ \ \ \mbox{for}\ \ x\in X_\zeta.
\ee

\begin{theorem} \label{th-Adam} Let $\zeta$ be a primitive substitution, satisfying the assumptions of Proposition~\ref{prop-zero}.
Suppose that $F = \sum_{a\in \Ak} d_a \One_{[a]}$, with $\int F\,d\mu=0$. Then the following holds for any $x\in X_\zeta$:
\be \label{eq-disc}
S_N^{x}(F) = O_{\zeta,\vec{s}} (1)\cdot ((\log_\theta N)^\alpha N^\beta),\ \ \mbox{with}\ \  \beta = \log_\theta |\theta_2|.
\ee
\end{theorem}

Adamczewski \cite{Adam} obtained (\ref{eq-disc}) for the case when $x$ is a fixed point of the substitution, but the extension to the case of general $x\in X_\zeta$ follows from
\cite[Lemma 3.2 ]{BuSo1} and \cite[Proposition 3.3]{BuSo1}.

\begin{proof}[Proof of Proposition~\ref{prop-zero}.]
We will use Lemma~\ref{lem-varr} for $\om=0$. We have $f = \sum_{a\in \Ak} f_a$, as in (\ref{fcyl}).
Clearly,
$$
|S_R^{(x,t)}(f,0) - S_R^{(x,0)}(f,0)| \le \|s\|_\infty \cdot \|f\|_\infty \le \|f\|_\infty.
$$
We can further find $N\in \Nat$ such that
$$
\left|x[0,N-1]\right|_{\vec{s}} \le R \le \left|x[0,N]\right|_{\vec{s}}.
$$
Then
$$
|S_R^{(x,0)}(f,0) - S_{R'}^{(x,0)}(f,0)| \le |R-R'|\cdot \|f\|_\infty \le \|s\|_\infty \cdot \|f\|_\infty \le \|f\|_\infty,
$$
where $R' = \left|x[0,N-1]\right|_{\vec{s}}$.
Thus it suffices to estimate $|S_{R'}^{(x,0)}(f,0)|$ from above. We have
$$
S_{R'}^{(x,0)}(f,0) = \sum_{a\in \Ak} \widehat{\psi}_a(0) \cdot \Phi_a^{\vec{s}} (x[0,N-1],0)
$$
from (\ref{SR1}). Observe that 
$$
\Phi_a^{\vec{s}} (x[0,N-1],0)= S_N^x(\One_{[a]}),
$$
according to (\ref{def-Phi3}) and (\ref{Birsum}), and $\widehat{\psi}_a(0) = \int_0^{s_a} \psi_a(t)\,dt$. Thus, $\int f \,d\wtil{\mu} = 0$ is equivalent to $\int F\,d\mu = 0$, where
$$F = \sum_{a\in\Ak} \widehat{\psi}_a(0)\cdot \One_{[a]},$$
by the definition of the measure $\wtil{\mu}$ on $\Xxi^{\vec{s}}$.
Since $R\ge N\cdot \min_{a\in \Ak} s_a$, we obtain from Theorem~\ref{th-Adam}
that 
$
S_R^{(x,t)}=O_{\zeta,\vec{s}} (1)\cdot ((\log_\theta R)^\alpha R^\beta),\ \ \mbox{with}\ \  \beta = \log_\theta |\theta_2|,
$
and Proposition~\ref{prop-zero} follows from Lemma~\ref{lem-varr}.
\end{proof}

\begin{proof}[Proof of Theorem~\ref{th-holder1}] It remains to ``glue'' Proposition~\ref{prop-holder2} with Proposition~\ref{prop-zero}. Fix $\wtil{\beta} \in (\beta,1)$. By Proposition~\ref{prop-zero},
we have $\sig_f([-r,r]) \le O_{\zeta,\vec{s}}(1) \cdot r^{2-2\wtil{\beta}}$, for any $r>0$. By Proposition~\ref{prop-holder2}, $\sig_f([\om-r,\om+r]) \le O_\zeta(1)\cdot  |\om|^{\wtil{\gamma}}$ for $r\le r_0 \cdot |\om|^Z$, where $r_0 = r_0(\vec{s})$. Keeping in mind that $Z>1$ without loss of generality, we obtain, for $r> r_0\cdot |\om|^Z$, the estimates:
\begin{multline}
\sig_f([\om-r,\om+r])  \le  \sig_f([-|\om|-r,|\om|+r])
                                  \le  O_{\zeta,\vec{s}}(1) \cdot (|\om|+r)^{2-2\wtil{\beta}} \\
                                                                           \le  O_{\zeta,\vec{s}}(1) \cdot (r^{1/Z}r_0^{-1/Z}+r)^{2-2\wtil{\beta}} 
                                                                             \le  O_{\zeta,\vec{s}}(1) \cdot r^{(2-2\wtil{\beta})/Z}.
\end{multline}
Now (\ref{holder1}) follows, with $\gam = \min\{\wtil{\gam},2-2\wtil{\beta})/Z\}$, and Theorem~\ref{th-holder1} is proved completely.
\end{proof}



{\bf Acknowledgements.} We are  deeply grateful to Jon Chaika for asking the question that led to this work  and for useful discussions.
The research of A. Bufetov on this project has received funding from the European Research Council (ERC) under the European Union's Horizon 2020 research and innovation programme under grant agreement No 647133 (ICHAOS).
It was also supported by the  Grant MD 5991.2016.1 of the President of the Russian Federation, by  the Russian Academic Excellence Project `5-100' and by the Gabriel Lam\'e Chair  at the Chebyshev Laboratory of the SPbSU, a joint initiative of the French Embassy in the Russian Federation and the Saint-Petersburg State University.  
B. Solomyak  has been supported  by Israel Science Foundation (grant 396/15).



\begin{thebibliography}{99}
\bibitem{Adam} Boris Adamczewski, Symbolic discrepancy and self-similar dynamics
{\em Annales Inst. Fourier},  {\bf 54}(7), 2201--2234.

\bibitem{ABBLS} Akiyama, S.; Barge, M.; Berth{\'e}, V.; Lee, J.-Y.; Siegel, A. On the Pisot substitution conjecture. Mathematics of aperiodic order, 33 --72, Progr. Math., 309, Birkhaeuser/Springer, Basel, 2015.

\bibitem{BBH} Xavier Bressaud, Alexander I. Bufetov, Pascal Hubert,
Deviation of ergodic averages for substitution dynamical systems with eigenvalues of modulus 1
{\em Proceedings of the London Mathematical Society},  {\bf 109} (2), 483 --522.

\bibitem{BD} M. Barge, B. Diamond,  Coincidence for substitutions of Pisot type, 
{\em Bulletin SMF} {\bf 130} (2002),
591 -- 626.

\bibitem{BuSo1} Alexander I. Bufetov and Boris Solomyak,
On the modulus of continuity for spectral measures in substitution dynamics, {\em Advances in Mathematics} {\bf 260} (2014), 84--129.

\bibitem{CSa} A. Clark and L. Sadun, When size matters: subshifts and their related tiling spaces, {\em Ergodic Theory Dynam.\ Systems} {\bf 23} (2003), 1043--1057.

\bibitem{Erd} Erd\H{o}s, Paul.
On the smoothness properties of Bernoulli convolutions.
{\em Amer.\ J.\ Math.} {\bf 62} (1940), 180--186.

\bibitem{FMN} S. Ferenczi, C. Mauduit, A. Nogueira, 
Substitution dynamical systems: algebraic characterization of eigenvalues, {\em Annales scientifiques de l'\'Ecole Normale Sup\'erieure,  S\'er.\ 4}, {\bf 29} no.\ 4 (1996), 519--533. 

\bibitem{Siegel} Pytheas N. Fogg,  {\em Substitutions in dynamics, arithmetics and combinatorics}, Edited by V. Berth\'e, S. Ferenczi, C. Mauduit and A. Siegel, in: Lecture Notes in Math., vol.\ 1794, Springer, Berlin, 2002.

\bibitem{HS} M. Hollander and B. Solomyak, Two-symbol Pisot substitutions have pure discrete spectrum,
{\em Ergodic Theory  Dynam.\ Systems}, {\bf  23} (02), 533 -- 540.

\bibitem{Kahane} J.\ P.\ Kahane, Sur la distribution de certaines series aleatoires, {\em
Colloque Th.\ Nombres [1969, Bordeaux],  Bull.\ Soc.\ math.\ France}, 
 M\'{e}moire 25 (1971), 119--122.
 
\bibitem{Queff} M. Queffelec, {\em Substitution Dynamical Systems - Spectral Analysis}. Second edition. Lecture Notes in Math., vol.\ 1294, Springer, Berlin, 2010.


\bibitem{solomyak}
Solomyak, Boris. On the spectral theory of adic transformations.
Representation theory and dynamical systems,  217--230, {\em Adv. Soviet Math.} {\bf 9},
Amer. Math. Soc., Providence, RI, 1992.

\bibitem{SolTil} B.~Solomyak, Dynamics of self-similar tilings, {\em Ergodic Theory Dynam.\ Systems} \textbf{17} (1997), no. 3, 695 -- 738.


\bibitem{Str90} 
 Strichartz, Robert S., Fourier asymptotics of fractal measures, {\em  J. Funct. Anal. } {\bf 89} (1990), no. 1, 154--187.
\end{thebibliography}
\end{document}